\documentclass[11pt]{amsart}
\usepackage{amsfonts}
\usepackage{amsmath}
\usepackage{amsthm}
\usepackage{url}
\usepackage[all]{xy}
\usepackage{latexsym}
\usepackage{amssymb}
\usepackage[cp850]{inputenc}
\usepackage{amsfonts}
\usepackage{graphicx}
\usepackage{dsfont}

\usepackage[usenames]{color}

\newtheorem{sat}{Theorem}[section]		
\newtheorem{lem}[sat]{Lemma}
\newtheorem{kor}[sat]{Corollary}			
\newtheorem{prop}[sat]{Proposition}

\newtheorem*{defi*}{Definition}			
\newtheorem*{bei*}{Example}
\newtheorem*{sat*}{Theorem}				
\newtheorem*{kor*}{Corollary}
\newtheorem*{rmk*}{Remark}				
	
\newtheorem*{quest*}{Question}	
	
\newtheorem{conj}{Conjecture}

\let\ssection=\section
\renewcommand{\section}{\setcounter{equation}{0}\ssection}

\newtheorem*{namedtheorem}{\theoremname}
\newcommand{\theoremname}{testing}
\newenvironment{named}[1]{\renewcommand{\theoremname}{#1}\begin{namedtheorem}}{\end{namedtheorem}}

\theoremstyle{remark}
\newtheorem*{bem}{Remark}

\newtheorem*{namedtheoremr}{\theoremnamer}
\newcommand{\theoremnamer}{testing}

\newcommand{\BR}{\mathbb R}			
			
\newcommand{\BS}{\mathbb S}

\newcommand{\CA}{\mathcal A}		
\newcommand{\CC}{\mathcal C}		
		
\newcommand{\CG}{\mathcal G}

\newcommand{\CM}{\mathcal M}		\newcommand{\CN}{\mathcal N}
\newcommand{\CO}{\mathcal O}		\newcommand{\CP}{\mathcal P}
\newcommand{\CQ}{\mathcal Q}		\newcommand{\CR}{\mathcal R}
\newcommand{\CS}{\mathcal S}		\newcommand{\CT}{\mathcal T}

\DeclareMathOperator{\vol}{vol}		%	Volumen

\DeclareMathOperator{\Map}{Map}
\DeclareMathOperator{\inj}{inj}
\DeclareMathOperator{\diam}{diam}

\newcommand{\comment}[1]{}

\DeclareMathOperator{\area}{area}

\DeclareMathOperator{\Fill}{Fill}
\DeclareMathOperator{\Lip}{Lip}

\DeclareMathOperator{\syst}{syst}

\DeclareMathOperator{\symlip}{SymLip}

\DeclareMathOperator{\arcsinh}{arcsinh}

\newcommand{\fsubd}{\mathrel{{\scriptstyle\searrow}\kern-1ex^d\kern0.5ex}}
\newcommand{\bsubd}{\mathrel{{\scriptstyle\swarrow}\kern-1.6ex^d\kern0.8ex}}

\newcommand{\BONE}{\mathds 1}

\begin{document}

\title[]{Counting curve types}
\author{Tarik Aougab}
\address{Department of Mathematics, Brown University}
\email{tarik\textunderscore  aougab@brown.edu}
\author{Juan Souto}
\address{IRMAR, Universit\'e de Rennes 1}
\email{jsoutoc@gmail.com}
\begin{abstract}
Let $S$ be a closed orientable hyperbolic surface, and let $\CO(K,S)$ denote the number of mapping class group orbits of curves on $S$ with at most $K$ self-intersections. Building on work of Sapir \cite{Sapir}, we give upper and lower bounds for $\CO(K,S)$ which are both exponential in $\sqrt{K}$.
\end{abstract}
\maketitle

\section{Introduction}

Let $S$ be a closed surface of genus $g\ge 2$. In this note we will be interested in the growth, as a function of $K$, of the number $\CO(K,S)$ of mapping class group orbits of curves $\gamma$ in $S$ with self-intersection number $\iota(\gamma,\gamma)\le K$. Recently, Sapir \cite{Sapir} proved that
\begin{equation}\label{eq-sapir}
\frac{1}{12} 2^{\sqrt{K}/12} \le \CO(K,S) \le (d_{S} \cdot \sqrt{K})^{d_{S} \sqrt{K}},
\end{equation}
where $d_{S}$ is a constant depending only on $S$. Our goal is to obtain, for $K$ large, a slightly improved lower exponential bound together with an also exponential upper bound. We show:

\begin{sat}\label{main}
For every $\delta>0$ there is $K_{\delta,S}$ with 
$$e^{(\pi\sqrt{|\chi(S)|}-\delta) \sqrt{K}} \le\CO(K,S)\le e^{(4\sqrt{2|\chi(S)|}+\delta)\sqrt{K}}$$
for every $K\ge K_{\delta,S}$.
\end{sat}

\begin{bem}
Note that Theorem \ref{main} does not say anything about the number $\CO(K,S)$ for $K$ small. See \cite{CFP} for some results in that direction. 
\end{bem}

We briefly comment on the proof of Theorem \ref{main}. While Sapir's methods are largely combinatorial, we obtain the lower bound using a probabilistic approach. We start with a result of Lalley \cite{Lalley-inter} asserting that the number of self-intersections $\iota(\gamma, \gamma)$ of a random geodesic $\gamma$  is essentially proportional to the square of its length. We obtain the desired lower bound from the facts that the number of geodesics of length $\le L$ grows like $\frac{e^L}L$, and that the number of times that the mapping class group orbit of a generic of length $\le L$ meets the set of all curves of length $\le L$ is bounded above by a polynomial in $L$.

To obtain the upper bound we associate to every curve $\gamma$ with $\iota(\gamma,\gamma)\le K$ a hyperbolic metric $\sigma_{\gamma}$ such that the $\sigma_\gamma$-geodesic corresponding to $\gamma$ has length at most $c_{S} \cdot \sqrt{K}$, for $c_{S}$ a constant depending only on $S$. It follows that, at least as long as $\gamma$ is filling, $\sigma_{\gamma}$ has injectivity radius at least $\ge e^{-c_S\cdot\sqrt{K}}$. We obtain the desired upper bound by approximating the $e^{-c_S\cdot\sqrt{K}}$-thick part of moduli space by a $\delta$-net $\CC$ (for some small $\delta>0$) whose cardinality grows polynomially with $K$. It follows that for each $\gamma$ with $\iota(\gamma, \gamma)\leq K$, $\gamma$ can be realized on one of the surfaces in the net $\CC$ with length roughly $\le c_{S}\cdot \sqrt{K}$, and thus we obtain the desired bound by considering all curves with length $ \leq c_{S} \cdot \sqrt{K}$ on any of the points in the $\delta$-net.

A metric $\sigma_\gamma$ with the needed properties has been constructed by the first author, Gaster, Patel, and Sapir \cite{AGPS}. However, if we were to use the metric provided by these authors, we would get an exponent in the upper bound which would be growing faster than $\sqrt{g}$ when we change $g$. This is why we provided an alternative construction, using circle packings, of the desired hyperbolic metric. We prove:

\begin{sat} \label{good metric}  
Let $S$ be a surface of finite topological type and with $\chi(S)<0$. For every closed curve $\gamma$ there is a hyperbolic metric $\rho$ on $S$ with respect to which the geodesic homotopic to $\gamma$ has length bounded by
\[ \ell_{\rho}(\gamma) \leq 4 \sqrt{2|\chi(S)| \cdot \iota(\gamma,\gamma)}.\]
\end{sat}

\begin{bem} Following the notation of \cite{Gas}, define $\overline{m}_{K}(S)$ to be the max, taken over all curves $\gamma$ on $S$ with $K$ self-intersections, of the infimal length of $\gamma$ over all points in Teichm{\"u}ller space $\CT(S)$. That is,
$$ \overline{m}_{K}(S):= \max \left\{ \inf \left\{ \ell_{\rho}(\gamma): \rho \in \CT(S) \right\}: \iota(\gamma, \gamma) = K. \right\} $$

One interpretation of Theorem \ref{good metric} is that it provides an upper bound for $\overline{m}_{K}(S)$. Then let $\overline{M}_{K}$ denote the supremum, taken over all surfaces $S$ of finite type with $\chi(S)<0$, of $\overline{m}_{K}(S)$. Gaster proves that $\overline{M}_{K}$ grows at least linearly in $K$ \cite{Gas}, and here we remark that Theorem \ref{good metric} provides the corresponding linear upper bound. Indeed, if $\iota(\gamma, \gamma) =K$ on any surface, the absolute value of the Euler characteristic of the subsurface it fills is at most $K$. Thus, Theorem \ref{good metric} yields a hyperbolic metric with respect to which $\gamma$ has length $\le 4 \sqrt{2K \cdot K}= 4\sqrt{2} \cdot K$. We record this together with Gaster's lower bound as follows:

\begin{kor} \label{Gas} 
$$\frac{\log(3)}{3} \leq \frac{\overline{M}_{K}}{K} \leq  4 \sqrt{2}. $$

\end{kor}

\end{bem}

The paper is organised as follows. In section \ref{sec-lalley} we recall Lalley's theorems on random geodesics, and in section \ref{sec-lower bound} we use these results to obtain the lower bound in Theorem \ref{main}.  In section \ref{sec-metric} we associate to each $\gamma$ a hyperbolic metric on $S$ satisfying the conclusion of Theorem \ref{good metric}. In section \ref{sec-net} we bound the cardinality of some maximal $\delta$-net in the $\epsilon$-thick part of moduli space with respect to the symmetric Lipschitz metric. In section \ref{sec-upper} we complete the proof of Theorem \ref{main}, and we conclude with some comments and observations in section \ref{sec-final} which will hopefully amuse the reader.

\section{The liouville current and random geodesics}\label{sec-lalley}
In this section we recall some facts about random geodesics. We will formulate our results using the language of geodesic currents. We refer the reader to \cite{Bonahon86,Bonahon88,Javi-Chris} for background on currents.

As in the introduction, let $S$ be a closed surface of genus $g\ge 2$ and $X$ a hyperbolic surface with underlying topological surface $S$. Denote by $\CC=\CC_X$ the space of all currents on $X$, endowed with the weak-*-topology. Every closed geodesic in $X$ determines a current. In this way we can identify the set $\CS$ of all homotopy classes of primitive closed essential curves in $X$ as a subset of $\CC$. In fact, the set consisting of all positive multiples of elements in $\CS$ is dense in $\CC$. 

The first key fact we will need about $\CC$ is that the function on $\CS$ which associates to each curve the length in $X$ of the corresponding geodesic extends to a continuous function, the {\em length function}
$$\ell_X:\CC\to\BR.$$
The function $\ell_X$ is homogenous under the action of $\BR_+$, meaning that 
$$\ell_X(t\cdot\lambda)=t\cdot\ell_X(\lambda)$$
for $\lambda\in\CC$ and $t>0$. In particular, we can identify the set 
$$\CC^1_X=\{\lambda\in\CC\vert\ell_X(\lambda)=1\}$$
of unit length currents with the space $P\CC$ of projective currents. This implies that $\CC^1_X$ is compact. 

Similarly, the function on $\CS\times\CS$ which associates to a pair $(\gamma,\eta)$ of curves their geometric intersection number $\iota(\gamma,\eta)$ extends continuously to the so-called intersection form
$$\iota:\CC\times\CC\to\BR.$$
The intersection form is homogenous on both factors: $\iota(t\cdot\lambda,s\cdot\mu)=st\cdot\iota(\lambda,\mu)$.

There is a particularly important current associated to the hyperbolic metric on $X$, called the \textit{Liouville current} and denoted here by $\lambda_X$. A defining property of $\lambda_{X}$ is that it links the length function $\ell_X$ and the intersection form $\iota(\cdot,\cdot)$. More precisely, the length of any current is the same as its intersection with $\lambda_X$:
$$\ell_X(\lambda)=\iota(\lambda,\lambda_X)\text{ for all }\lambda\in\CC.$$
In particular we have
$$\ell_X(\lambda_X)=\iota(\lambda_X,\lambda_X)=\pi^2|\chi(S)|,$$
where we obtain the last equality from \cite{Bonahon88}. We denote by
$$\lambda_X^1=\frac 1{\pi^2|\chi(S)|}\lambda_X\in\CC^1_X,$$
the unit length current associated to $\lambda_X$.

Unit length currents can be interpreted as probability measures on $T^1X$ invariant under the geodesic flow. From this point of view, the Liouville current arises from the Liouville measure, which is the measure of maximal entropy of the geodesic flow. In particular, it follows from the work of Lalley \cite{Lalley-limit} that randomly chosen geodesics on $X$ converge, once considered as currents, to $\lambda_X$. 

To make this precise, let $\CS_X(L)$ be the set of all geodesics in $X$ of length $\le L$ and recall that by Huber's theorem \cite{Buser} its cardinality behaves like
\begin{equation}\label{eq-margulis}
\vert\CS_X(L)\vert\sim\frac{e^L}L,
\end{equation}
meaning that the ratio between both quantities tends to $1$. Now, reinterpreting a theorem of Lalley \cite{Lalley-limit} one sees that as $L$ grows, the measures 
$$\sigma_L=\frac 1{\vert\CS_X(L)\vert}\sum_{\gamma\in\CS_X(L)}\delta_{\frac 1{\ell_X(\gamma)}\gamma}$$
on $\CC^1(X)$ converge in the weak-*-topology on the space of measures to the Liouville measure $\lambda_X^1$, that is
\begin{equation}\label{eq-lalley}
\lim_{L\to\infty}\sigma_L=\delta_{\lambda_X^1}.
\end{equation}
Here $\delta_c$ is the Dirac measure on $\CC(X)$ centered on the current $c$.

Note that combining \eqref{eq-lalley} and \eqref{eq-margulis} we get:

\begin{lem}\label{lem-lalley1}
We have
$$\left\vert\left\{\frac 1{\ell_X(\gamma)}\gamma\in U\middle\vert \ell_X(\gamma)\le L\right\}\right\vert\sim\frac{e^L}L$$
for every open neighborhood $U\subset\CC^1_X$ of $\lambda_X^1$.\qed
\end{lem}

A corollary of Lalley's formula \eqref{eq-lalley} is that the number of intersections of a randomly chosen geodesic is basically proportional to the square of the length \cite{Lalley-note}. More precisely, Lalley proves that for every $\delta>0$ we have
$$\lim_{L\to\infty}\frac 1{\vert\CS_X(L)\vert}\left\vert\left\{\gamma\in\CS_X(L)\text{ with }\left\vert\frac{\iota(\gamma,\gamma)}{\ell_X(\gamma)^2}-\iota(\lambda_X^1,\lambda_X^1)\right\vert\ge\delta
\right\}\right\vert=0$$
Again, combining this result with \eqref{eq-margulis} and denoting
\begin{equation}\label{eq-defik}
\iota(\lambda_X^1,\lambda_X^1)=\frac 1{\pi^2|\chi(S)|}\stackrel{\text{def}}=:\kappa
\end{equation}
we get:

\begin{lem}\label{lem-lalley2}
Given a hyperbolic surface $X$ and $\delta>0$ consider the set
$$\CA(X,\delta)=\left\{\gamma\in\CS\text{ with }\left\vert\frac{\iota(\gamma,\gamma)}{\ell_X(\gamma)^2}-\kappa\right\vert\le\delta\right\},$$
where $\kappa$ is as in \eqref{eq-defik}. Then we have that
$$\left\vert\left\{\gamma\in\CA(C,\delta)\vert \ell_X(\gamma)\le L\right\}\right\vert\sim\frac{e^L}L$$
for every $\delta>0$.\qed
\end{lem}

\section{The lower bound}\label{sec-lower bound}

Recall that $\CO(K,S)$ is the number of mapping class group orbits of curves $\gamma$ in $S$ with self-intersection number $\iota(\gamma,\gamma)\le K$. In this section we prove: 

\begin{prop}\label{sat-lower bound}
Let $S$ be a closed surface with $\chi(S)<0$. For every $\delta>0$ there is $K_{\delta,S}$ with 
$$\CO(K,S)\ge e^{(\sqrt{\pi^2|\chi(S)|}-\delta)\sqrt K}$$
for every $K\ge K_{\delta,S}$.
\end{prop}

Before launching into the proof of this proposition, we require two lemmas. Notation is as in the previous section.

\begin{lem}\label{lem1}
Let $\alpha\in\CC$ be a filling current. Then there is $\epsilon>0$ such that the set of curves
$$\CG(X,\alpha,\epsilon)=\left\{\gamma\in\CS\middle\vert
\frac{\iota(\gamma,\lambda)}{\ell_X(\gamma)\cdot\iota(\alpha,\lambda)}>\epsilon\text{ for all }\lambda\in\CC\right\}$$
satisfies $\vert\{\gamma\in\CG(X,\alpha,\epsilon)\vert\ell_X(\gamma)\le L\}\vert\sim\frac{e^L}L.$
\end{lem}

Recall that a current $\alpha$ is filling if it has positive intersection number $\iota(\alpha,\mu)$ with all non-zero currents $\mu$. 

\begin{proof}
Consider the continuous map
\begin{equation}\label{eq-I need a coffee}
\CC^1_X\times\CC^1_X\to\BR,\ \ (\mu,\lambda)\mapsto\frac{\iota(\mu,\lambda)}{\iota(\alpha,\lambda)},
\end{equation}
and note that, since the normalized Liouville current $\lambda_X^1$ of $X$ is filling, and since $\CC^1_X$ is compact, the quantity
$$\epsilon=\frac 12\min_{\lambda\in\CC^1_X}\frac{\iota(\lambda_X^1,\lambda)}{\iota(\alpha,\lambda)}$$
is positive. Now, continuity of \eqref{eq-I need a coffee} together with the compactness of the domain imply that there is an open neighborhood $U$ of $\lambda_X^1$ in $\CC^1$ with 
$$\frac{\iota(\mu,\lambda)}{\iota(\alpha,\lambda)}\ge\epsilon$$
for every $\mu\in U$ and every $\lambda\in\CC^1$. Note that this implies that
\begin{equation}\label{eq-still need coffee1}
\left\{\gamma\in\CS(X)\middle\vert\frac 1{\ell_X(\gamma)}\gamma\in U\right\}\subset\CG(X,\alpha,\epsilon).
\end{equation}
The claim now follows from Lemma \ref{lem-lalley1} and \eqref{eq-margulis}.
\end{proof}

\begin{lem}\label{lem2}
Given a hyperbolic surface $X$, a full marking $\alpha$ and $\epsilon>0$, there is a constant $C$ with 
$$\vert \left\{\eta\in\Map(X)\cdot\gamma\vert\ell_X(\eta)\le L\right\}\vert\le C\cdot L^C,$$
for every $L$ and every $\gamma\in\CG(X,\alpha,\epsilon)$.
\end{lem}

Recall that a {\em full marking} is nothing other than a pants decomposition plus a transversal curve for every component of the pants decomposition. As a current, a full marking is filling.

\begin{proof}
Let $d_{\Lip}$ stand for the Lipschitz metric on Teichm\"uller space $\CT(X)$ of $X$ and recall that there is a constant $K= K(\alpha, X)$ with
$$d_{\text{Lip}}(X,Y)\le K\log(\ell_Y(\alpha))+K$$
for any $Y\in\CT(X)$ \cite[Theorem E]{LRT}. 

Suppose now that we have $\phi\in\Map(X)$ a mapping class and let $\gamma\in\CG(X,\alpha,\epsilon)$. We then have
\begin{align*}
\ell_X(\phi(\gamma))
&=\iota(\lambda_X,\phi(\gamma))=\iota(\lambda_{\phi^{-1}(X)},\gamma)\ge \epsilon\cdot\ell_X(\gamma)\cdot\iota(\lambda{\phi^{-1}(X)},\alpha)\\
&= \epsilon\cdot\ell_X(\gamma)\cdot\iota(\lambda_X,\phi(\alpha))= \epsilon\cdot\ell_X(\gamma)\cdot\ell_{\phi^{-1}(X)}(\alpha).
\end{align*}
In particular we have
$$d_{\Lip}(X,\phi^{-1}(X))\le K\log\left(\frac{\ell_X(\phi(\gamma))}{\epsilon\cdot\ell_X(\gamma)}\right)+K$$
and thus, up to increasing $K$ by some fixed amount (depending only on $\epsilon$ and $X$), we have
$$d_{\Lip}(X,\phi^{-1}(X))\le K\log\left(\ell_X(\phi(\gamma))\right)+K.$$
Altogether, it follows that the set $\left\{\eta\in\Map(X)\cdot\gamma\vert\ell_X(\eta)\le L\right\}$ has at most as many elements as the set
\begin{equation}\label{eq-set}
\left\{\phi\in\Map(X)\vert d_{\Lip}(X,\phi^{-1}(X))\le K\log L+K\right\}.
\end{equation}
Now, it is known that the orbit of a point in Teichm\"uller space under the action of the mapping class group grows exponentially - to see that this is the case note for example that, in the thick part the Lipschitz metric and the Teichm\"uller metric are comparable and the Teichm\"uller metric has volume growth entropy $6g-6$ \cite{ABEM}. It follows that there is a constant $K'$ such that the set \eqref{eq-set} has at most $K'\cdot e^{K'\log L+K'}$ elements. Since $K$ and $K'$ are independent of $L$ and $\gamma\in\CG(X,\alpha,\epsilon)$ we obtain the existence of a constant $C$ which is independent of both $L$ and $\gamma$, so that 
$$\vert \left\{\eta\in\Map(X)\cdot\gamma\vert\ell_X(\eta)\le L\right\}\vert\le C\cdot L^C,$$
as desired.
\end{proof}

We are now ready for the proof of Proposition \ref{sat-lower bound}:

\begin{proof}[Proof of Proposition \ref{sat-lower bound}]
Let $\alpha$ be a full marking, $\epsilon>0$ such that 
$$\vert\{\gamma\in\CG(X,\alpha,\epsilon)\vert\ell_X(\gamma)\le L\}\vert\sim\frac{e^L}L,$$
where $\CG=\CG(X,\alpha,\epsilon)$ is as in the statement of Lemma \ref{lem1}, and $\delta$ positive and arbitrary. Let $\CA=\CA(X,\delta)$ be as in Lemma \ref{lem-lalley2} and recall that by said lemma we have
$$\vert\{\gamma\in\CA(X,\delta)\vert\ell_X(\gamma)\le L\}\vert\sim\frac{e^L}L.$$
The desired lower bound for $\CO(K,S)$ will arise from counting mapping class group orbits of elements represented by curves in $\CG\cap\CA$. Lemma \ref{lem1} asserts that the mapping class group orbit $\Map(X)\cdot\gamma$ of $\gamma\in\left\{\gamma\in\CG\cap\CA\vert \ell_X(\gamma)\le L\right\}$ meets this set at most $C\cdot L^C$ times. This implies that $\left\{\gamma\in\CG\cap\CA\vert \ell_X(\gamma)\le L\right\}$ meets at least $\frac{e^L}{C\cdot L^{C+1}}$ distinct mapping class group orbits.

Now, setting $\kappa=\iota(\lambda_X^1,\lambda_X^1)=\frac 1{\pi^2|\chi(S)|}$ as in \eqref{eq-defik}, note that for all $\gamma\in\CA(X,\delta)$ we have
$$\sqrt{\frac{\iota(\gamma,\gamma)}{\kappa+\delta}}\le\ell_X(\gamma)\le\sqrt{\frac{\iota(\gamma,\gamma)}{\kappa-\delta}}.$$
It follows that the set
$$\left\{\gamma\in\CG\cap\CA\middle\vert \sqrt{\frac{\iota(\gamma,\gamma)}{\kappa+\delta}}\le L\right\}$$
meets at least $\frac{e^L}{C\cdot L^{C+1}}\ge\frac{e^{\sqrt{\frac{\iota(\gamma,\gamma)}{\kappa+\delta}}}}{C\cdot\sqrt{\frac{\iota(\gamma,\gamma)}{\kappa-\delta}}^{C+1}}$
distinct mapping class group orbits. The proposition follows after some elementary algebra and possibly choosing a new $\delta$.
\end{proof}

\section{Finding a suitable metric}\label{sec-metric}

In this section we prove Theorem \ref{good metric}. We recall the statement for the convenience of the reader:

\begin{named}{Theorem \ref{good metric}}  
Let $S$ be a surface of finite topological type and with $\chi(S)<0$. For every closed curve $\gamma$ there is a hyperbolic metric $\rho$ on $S$ with respect to which the geodesic homotopic to $\gamma$ has length bounded by
\[ \ell_{\rho}(\gamma) \leq 4 \sqrt{2|\chi(S)| \cdot \iota(\gamma,\gamma)}.\]
\end{named}

\begin{proof}
Place $\gamma$ in general position (and thus no triple points) and without bigons and consider its image as a graph $\Gamma$ on the topological surface $S$. Add (possibly ideal) edges to $\Gamma$ to obtain a triangulation $\hat\Gamma$. Now, $\hat\Gamma$ determines a topological circle packing of $S$; that is, circles are topological circles, each vertex corresponds to a circle and two circles are adjacent (meaning that they touch in a point) if and only if they are joined by an edge.

By Koebe's Discrete Uniformization theorem \cite[Theorem 4.3]{Stephenson}, there is thus a hyperbolic structure $\rho$ on $S$ with respect to which $\hat\Gamma$ is the dual graph of an actual circle packing. More concretely, we realize $\hat\Gamma$ in such a way that each vertex $v\in V=V(\hat\Gamma)$ goes to the center of the corresponding circle $C_v$ and let $r_v$ be its radius. Now note that
$$\ell_\rho(\gamma)\le\ell_\rho(\Gamma)\le 4\sum_{v\in V}r_v,$$
where the $4$ comes from the fact that $\gamma$ goes through each vertex twice.

Now, to estimate $\sum_v r_v$ consider the functions
$$\CR:V(\hat\Gamma)\to\BR,\ \ \CR(v)=r_v\text{ for all }v$$
$$\BONE:V(\hat\Gamma)\to\BR,\ \ \BONE(v)=1\text{ for all }v$$ 
as elements in $L^2=L^2(V(\hat\Gamma))$. Then we get from the Cauchy-Schwartz inequality that
$$\sum_{v\in V}r_v=\langle R,\BONE\rangle_{L^2}\le\Vert R\Vert_{L^2}\Vert\BONE\Vert_{L^2}$$
Now, note that 
$$\Vert\BONE\Vert_{L^2}=\sqrt{\langle\BONE,\BONE\rangle_{L^2}}=\sqrt{\sum_{v\in V}1}=\sqrt{\Vert V\vert}=\sqrt{\iota(\gamma,\gamma)}$$
On the other hand, for each circle $C_v$ we have
$$\area(C_v)=2\pi(\cosh(r_v)-1)\ge\pi r_v^2$$
Since all the disks are disjoint, and since $X$ has area $2\pi\vert\chi(X)\vert$ we then get that
$$2\pi\vert\chi(X)\vert=\vol(X)\ge\sum_{v\in V}\pi r_v^2=\pi\cdot\Vert R\Vert_{L^2}^2$$
Altogether we get that
$$\ell_\rho(\gamma)\le 4\sum_{v\in V}r_v\le 4 \Vert R\Vert_{L^2}\Vert\BONE\Vert_{L^2}\le 4\sqrt{2|\chi(S)|}\sqrt{\iota(\gamma,\gamma)},$$
as we needed to prove.
 \end{proof}

Recall now that a simple closed geodesic $\alpha$ in a hyperbolic surface $X$ has a collar of width at least 
$$\arcsinh\left(\frac 1{\sinh\left(\frac 12\ell_X(\alpha)\right)}\right)\ge \log\left(\frac 1{\ell_X(\alpha)}\right).$$
This implies that any filling curve $X$ has length at least $\log(\ell_X(\alpha)^{-1})$. In particular, if the curve $\gamma$ in Theorem \ref{good metric} is filling we see that the produced hyperbolic surface satisfies the following bound on injectivity radius:

\begin{kor}\label{kor-inj-bound}
With notation as in Theorem \ref{good metric}, suppose that $\gamma$ is filling. Then we have that
$$\ell_\rho(\alpha)\ge e^{-(4\sqrt{2|\chi(S)|\cdot\iota(\gamma,\gamma)})}$$
for every closed geodesic $\alpha$ in $(S,\rho)$.\qed
\end{kor}

For the sake of completeness we comment briefly on the case that $\gamma$ is not filling. In that case, one can use an argument taken from \cite{AGPS} to modify the construction above, obtaining a new hyperbolic metric $\rho'$ which still satisfies Corollary \ref{kor-inj-bound}, and such that for $\iota(\gamma, \gamma)$ sufficiently large (above some universal constant not depending on $S$), 
$$ \ell_{\rho'}(\gamma) \leq 4 \sqrt{2 |\chi(S)| \cdot \iota(\gamma, \gamma)} +1 . $$
In fact, for any $\epsilon >0$ there is $R(\epsilon)$ so that for $\iota(\gamma, \gamma)>R$, one has 
$$  \ell_{\rho'}(\gamma) \leq 4 \sqrt{2 |\chi(S)| \cdot \iota(\gamma, \gamma)} +\epsilon . $$

We sketch this as follows. Let $Y \subsetneq S$ denote the subsurface of $S$ filled by $\gamma$. Then note that the argument used above in the proof of Theorem \ref{good metric} applies directly to the surface $Y$, and produces a metric $\rho_{Y}$ on $Y$ assigning length $0$ to each boundary component of $Y$. Since $\ell(\gamma) \leq 4 \sqrt{2 |\chi(S)|\cdot\iota(\gamma,\gamma)}$, the geodesic representative for $\gamma$  can not penetrate more than $2 \sqrt{2| \chi(S)|\cdot\iota(\gamma,\gamma)}$ into any of the standard cusp neighborhoods. We exploit this by replacing each cusp with a geodesic boundary component of length $\sim e^{-\sqrt{\iota(\gamma,\gamma)}}$; by a standard geometric convergence argument, this can be done while barely changing the metric on the portion of $Y$ in which $\gamma$ resides. This produces a hyperbolic surface with totally geodesic boundary on which the length of $\gamma$ still satisfies the desired upper bound. Now we simply glue a sufficiently thick copy of the complementary subsurface to $Y$ over its boundary components to complete the construction.

\section{Bounding the size of nets}\label{sec-net}

In addition to the existence of the metric provided by Theorem \ref{good metric}, the proof of Theorem \ref{main} will rely on having some control on the size of an approximating net in moduli space. Let $S$ be a surface of finite topological type and $\chi(S)<0$, and consider both Teichm\"uller space $\CT(S)$ and moduli space $\CM(S)$ to be endowed with the {\em symmetric Lipschitz metric} 
$$d_{\symlip}(X,Y)=\max\{d_{\Lip}(X,Y),d_{\Lip}(Y,X)\},$$
and we are going to be interested in the number of points that we need to approximate the thick part of moduli space with respect to this metric. More concretely, given $\epsilon$ and $\delta$ positive let 
\begin{equation}\label{eq-net}
n_S(\epsilon,\delta)=\begin{cases}
\text{minimal cardinality of a }\delta\text{-dense}\\
\text{set in }(\CM^{\ge\epsilon}(S),d_{\symlip}),
\end{cases}
\end{equation}
where 
$$\CM^{\ge\epsilon}(S)=\{X\in\CM(S)\vert\syst(X)\ge\epsilon\}$$
is the set of hyperbolic structures of $S$ without geodesics shorter than $\epsilon$. We prove:

\begin{prop}\label{prop-net}
There is $C,N>0$ and a function $f:\BR_+\to\BR_+$ all depending only on $S$, so that  
$$n_S(\epsilon,\delta)\le C\cdot\left\vert\log(\epsilon)\right\vert^N\cdot f(\delta)$$
for all $\epsilon,\delta$.
\end{prop}

Before launching into the proof of Proposition \ref{prop-net}, we require some notation. Given a pants decomposition $P$ of $S$ let
$$\Phi^P:\BR_+^{3g-3}\times\BR^{3g-3}\to\CT(S)$$
be the corresponding Fenchel-Nielsen coordinates of Teichm\"uller space; we will be following the conventions in \cite{Buser}. Consider the subsets
$$\CQ^P=(0,26(g-1))^{3g-3}\times[0,1]^{3g-3}$$
$$\CQ^P_\epsilon=(\epsilon,26(g-1))^{3g-3}\times[0,1]^{3g-3}$$
of $\BR_+^{3g-3}\times\BR^{3g-3}$ and, abusing notation, identify $\CQ^P=\Phi^P(\CQ^P)$ and accordingly identify $\CQ^P_\epsilon=\Phi^P(\CQ^P_\epsilon)$. Note also that when varying $X\in\CQ^P$, the length $\ell_X(\gamma)$ of a curve $\gamma$ significantly increases only if the curves in $P$ are becoming very short. More concretely, we get that for every curve $\gamma\subset S$ there is some $c^P_\gamma$ with
\begin{equation}\label{eq-length-boundFN}
\ell_X(\gamma)\le c^P_\gamma\cdot\vert\log\epsilon\vert+c_\gamma^P
\end{equation}
for every $X\in\CQ^P_\epsilon$.

The proof of Proposition \ref{prop-net} will rely on the following estimate for the symmetric Lipschitz distance between points in $\CQ^P$.

\begin{lem}\label{lem-finite collection}
There is a finite collection $\Gamma$ of simple closed curves in $S$ such that for all $\delta>0$ and $X,Y \in \CQ^P$ there is $\delta_0$ so that
$$\max_{\gamma\in\Gamma}\vert\log(\ell_X(\gamma))-\log(\ell_Y(\gamma))\vert\le\delta_0  \Rightarrow d_{\symlip}(X,Y)\le\delta.$$
\end{lem}

Before we address the proof of Lemma \ref{lem-finite collection} we make an observation which will come in handy over the course of the argument. Fixing some $\eta$ positive, let $\epsilon \ll \eta$ also positive, and let $C(\epsilon)$ be the hyperbolic cylinder with soul of length $\epsilon$ and whose boundary components both have constant curvature and length $\eta$. We parameterise $C(\epsilon)=\BS^1\times[0,1]$ in such a way that for each $\theta\in\BS^1$ the segment $t\to (\theta,t)$ is a parametrised minimal length geodesic segment between both boundary components and, (2) for each $t$ the circle $\theta\to(\theta,t)$ has constant curvature. 

Suppose that we are now given two such hyperbolic cylinders $C(\epsilon)=\BS^1\times[0,1]$ and $C(\epsilon')=\BS^1\times[0,1]$, and that for some $\alpha\in\BS^1$ we consider the map
\begin{equation}\label{eq-map model}
f_\alpha:C(\epsilon)\to C(\epsilon'), \ \ f_\alpha(\theta,t)\mapsto(\theta+t\cdot\alpha,t).
\end{equation}
A simple, but not very elegant computation yields the following bound for the Lipschitz constant of $f_\alpha$:
\begin{equation}\label{eq-lip model}
\Lip(f_\alpha)\le(1+\epsilon')\max\left\{\frac{\epsilon'}\epsilon,\frac{\log\epsilon'}{\log\epsilon}\right\}
\end{equation}
whenever $\epsilon,\epsilon'$ are smaller than some universal constant $\epsilon_0$ - note that the bound does not depend on the twist $\alpha\in\BS^1$. Note also that, this bound can be made very close to $1$ if (1) $\epsilon$ and $\epsilon'$ are small and (2) the absolute value $\vert\log(\epsilon)-\log(\epsilon')\vert$ of the difference of the logarithms is also small. We now prove Lemma \ref{lem-finite collection}:

\begin{proof}[Proof of Lemma \ref{lem-finite collection}]
Recall that for every finite type surface $\Sigma$ there is a finite collection of curves $\Gamma_\Sigma$ such that the map
$$\CT(\Sigma)\to\BR^{\Gamma_\Sigma},\ \ X\mapsto (\log(\ell_X(\gamma)))_{\gamma\in\Gamma_\Sigma}$$
is injective \cite{Buser}. We let $\Gamma$ be a collection of curves which contains $P$ and the collection $\Gamma_{S\setminus P'}$ for every subset $P'\subset P$. We claim that this collection $\Gamma$ satisfies the claim.

Note now that it suffices to prove that whenever we are given sequences $(X_i),(Y_i)$ in $\CQ=\CQ^P$ with 
\begin{equation}\label{eq-blabla}
\vert\log(\ell_{X_i}(\gamma))-\log(\ell_{Y_i}(\gamma))\vert\to 0
\end{equation}
for all $\gamma\in\Gamma$, then we have $d_{\symlip}(X_i,Y_i)\to 0$. In the context of \eqref{eq-blabla}, assume first that the sequences $(X_i)$ and $(Y_i)$ converge to some $X,Y\in\CQ$. Then we have that $\ell_X(\gamma)=\ell_Y(\gamma)$ for all $\gamma\in\Gamma$ and thus that $X=Y$ because $\Gamma_S\subset\Gamma$. This means that $d_{\symlip}(X_i,Y_i)\to 0$, as desired.

Otherwise, suppose that the sequence $(X_i)$ diverges in $\CQ$, and note that this is only possible if some of the curves in $P$ are being pinched. Passing to a subsequence we can assume that there is thus a subcollection $P'\subset P$ such that $\ell_{X_i}(\gamma)\to 0$ if and only if $\gamma\in P'$. Note then that, since $P\subset\Gamma$, we also get from \eqref{eq-blabla} that $\ell_{Y_i}(\gamma)\to 0$ if and only if $\gamma\in P'$.

 It follows that, choosing $\eta$ positive and small enough, after passing again to subsequences, we can assume that the thick parts $X_i^{\ge\eta}$ and $Y_i^{\ge\eta}$ converge geometrically to the thick parts $X_\infty^{\ge\eta}$ and $Y_\infty^{\ge\eta}$ of complete hyperbolic structures $X_\infty$ and $Y_\infty$ on $S\setminus P'$. We get from \eqref{eq-blabla} that $\ell_{X_\infty}(\gamma)=\ell_{Y_\infty}(\gamma)$ for every $\gamma\in\Gamma$ contained in $S\setminus P'$. Since $\Gamma_{S\setminus P'}\subset\Gamma$, it follows that $X_\infty=Y_\infty$. 

Since the thick parts of $X_i$ and $Y_i$ converge geometrically to the same limit $X_\infty=Y_\infty$, it follows that there are maps
$$\phi_i:X_i\to Y_i$$
in the correct homotopy class which, when $i$ grows, induce more and more isometric maps $\phi_i\vert_{X_i^{\ge\eta}}:X_i^{\ge\eta}\to Y_i^{\ge\eta}$. Note that we can assume without loss of generality that $\phi_i$ is actually isometric on the boundary of $X_i^{\ge\eta}$. 

Note also that the $\eta$-thin parts $X_i^{\le\eta}$ and $Y_i^{\le\eta}$ are disjoint unions of cylinders like those considered in the remark before the proof, and we can homotope $\phi$ on its restriction to each component of the thin part so that on each such cylinder, it is of the form \eqref{eq-map model}. Moreover, since $P\subset\Gamma$ we obtain from \eqref{eq-blabla} and \eqref{eq-lip model} that the Lipschitz constant of the induced map between thin parts is arbitrarily close to $1$. Altogether we have that the map $\phi_i:X_i\to Y_i$ is homotopic to a map with Lipschitz constant $L_i\to 1$. Thus, $d_{\Lip}(X_i,Y_i)\to 0$. Since the whole argument is symmetric, we deduce that the same is true if we reverse the roles of $X_i$ and $Y_i$. This yields that $d_{\symlip}(X_i,Y_i)\to 0$ concluding the proof of Lemma \ref{lem-finite collection}.
\end{proof}

Armed with Lemma \ref{lem-finite collection}, we can conclude the proof of Proposition \ref{prop-net}:

\begin{proof}[Proof of Proposition \ref{prop-net}]
Recall that it is a theorem of Bers (see \cite{Buser}) that every surface in Teichm\"uller space admits a pants decomposition whose curves have length at most $26(g-1)$. It follows that if $P_1,\dots,P_s$ are pants decompositions of $S$ such that every pants decomposition is mapping class group equivalent to one of those then we have that
$$\CQ^{P_1}\cup\dots\cup\CQ^{P_s}$$
is a coarse fundamental domain for the action of the mapping class group on Teichm\"uller space. It follows that to find a $\delta$-dense set in $\CM^{\ge\epsilon}$ it suffices to find a $\delta$-dense for each one of the sets $\CQ^{P_i}_\epsilon$. We state what we have to prove in these terms:
\medskip

\noindent{\bf Claim.} {\em Let $P$ a pants decomposition of $S$. There are $C,N,\epsilon_0>0$ and a function $f:\BR_+\to\BR_+$ such that for all positive $\epsilon<\epsilon_0$ and $\delta$ there is a $\delta$-dense set in $(\CQ^P_\epsilon,d_{\symlip})$ with at most $C\cdot\left\vert\log(\epsilon)\right\vert^N\cdot f(\delta)$ elements.}
\medskip

It remains to prove the claim. Since the pants decomposition is now fixed, we drop every reference to it from our notation. Let $\Gamma$ be the collection of curves provided by Lemma \ref{lem-finite collection} and consider the map
$$\lambda:\CT(S)\to\BR^\Gamma,\ \ X\mapsto(\log\ell_X(\gamma))_{\gamma\in\Gamma}.$$
We endow the domain $\CT(S)$ with $d_{\symlip}$ and the image $\BR^\Gamma$ with the supremum norm $\Vert\cdot\Vert_\infty$. Note also that by \eqref{eq-length-boundFN} there is some $c=c(\Gamma)$ such that for every $\epsilon$ we have
\begin{equation}\label{eq-box}
\lambda(\CQ_\epsilon)\subset[c\log(\epsilon)-c,c\vert\log\epsilon\vert+c]^\Gamma.
\end{equation}
Suppose now that we are given $\delta$ and set $f(\delta)=\delta_0$ where the latter is the constant provided by Lemma \ref{lem-finite collection}.

Then, given \eqref{eq-box}, a packing argument in euclidean space implies that $(\lambda(\CQ_\epsilon),\Vert\cdot\Vert_\infty)$ has a $\delta_0$-dense set $\CN$ with at most 
$$(2c\vert\log\epsilon\vert+2c)^{\vert \Gamma\vert}\delta_0^{-\vert\Gamma\vert}$$
elements. Now, lemma \ref{lem-finite collection} shows that the set $\lambda^{-1}(\CN)\subset\CQ_\epsilon$ is $\delta$-dense with respect to $d_{\symlip}$. The claim follows.
\end{proof}

\section{The upper bound }\label{sec-upper}

In this section we give upper bounds for the number of mapping class group orbits of curves with at most $K$ self-intersections and conclude the proof of Theorem \ref{main}. We will however first consider an auxiliary quantity. Given a surface $S$ of finite type with $\chi(S)<0$, and given $\epsilon$ and $L$ positive, let $\CS(S,\epsilon,L)$ be the number of all $\Map(S)$-orbits of filling curves $\gamma\subset S$ with the property that there is $(S,\rho)\in\CM^{\ge\epsilon}(S)$ such that $\ell_\rho(\gamma)\le L$. Using the results of the previous sections we bound $\CS(S,\epsilon,L)$ as follows:

\begin{prop}\label{prop-upper-bound-thick}
For every $S$ and $\delta$ there is $C$ such that for all $\epsilon$ and $L$ we have
$$\CS(S,\epsilon,L)\le C\cdot n_S(\epsilon,\delta)\cdot e^{e^{\delta}L}.$$
where $C$ is a constant which depends only on the topology of the surface and where $n_S(\epsilon,\delta)$ is as in \eqref{eq-net}.
\end{prop}
\begin{proof}
Fixing $\epsilon$, let $\CN$ be a $\delta$-dense set in $(\CM^{\ge\epsilon}(S),d_{\symlip})$ with cardinality $n_S(\epsilon,\delta)$, and note that for every point $X\in\CM^{\ge\epsilon}$ there is $Y\in\CN$ and an $e^{\delta}$-lipschitz map $X\to Y$. It follows that each mapping class orbit contributing to $\CS(\epsilon,L)$ is represented by some curve which has length at most $e^{\delta}L$ with respect to some $Y\in\CN$. In other words, we have
$$\CS(S,\epsilon,L)\le\sum_{Y\in\CN}\vert\CS_Y(e^{\delta}L)\vert$$
where $\CS_Y(L)$ is, as in section \ref{sec-lalley}, the set of all curves which have length at most $L$ in $Y$.

Fix now a small number like $\mu=\frac 1{10}$. The $\mu$-thin part of each $Y\in\CN$ has at most $\vert\chi(S)\vert$ connected components. Choose a base point in each one of those $\mu$-thick parts and let $\pi_Y\subset Y$ be the set consisting of those points. The diameter of each component of the $\mu$-thick part of the surface $Y$ is bounded from above by 
$$\diam\le\frac{2\vert\chi(S)\vert}{\mu^2}=200\vert\chi(S)\vert.$$
Since all curves in $\CS(S, \epsilon, L)$ are filling, they enter some thick part. It follows that each curve in $\CS(S,\epsilon,L)$ is represented on at least one surface $Y\in\CN$ by a curve of length $\le e^{2\delta}L+400\vert\chi(S)\vert$ which passes through one of the marked points in the set $\pi_Y$. 

Now, if we are given $Y\in\CN$ and $x\in\pi_Y$ we have the upper bound
$$\le 100\cdot e^{e^{2\delta}L+400\vert\chi(S)\vert}$$
for the cardinality of the set of loops of length at most $e^{2\delta}L+400\vert\chi(S)\vert$ that pass through the $\frac 1{10}$-thick point $x$. Since there are at most $\vert\chi(S)\vert$ choices for $x$ and since there are $n_S(\epsilon,\delta)$ choices for $Y$, we get that
$$\CS(S,\epsilon,L)\le 100\cdot \vert\chi(S)\vert\cdot n_S(\epsilon,\delta)\cdot e^{e^{2\delta}L+400\vert\chi(S)\vert}$$
which is what we wanted to prove.
\end{proof}

Armed with Theorem \ref{good metric}, Proposition \ref{prop-net} and Proposition \ref{prop-upper-bound-thick}, we are ready to prove Theorem \ref{main}:

\begin{proof}[Proof of  Theorem \ref{main}]
The lower bound comes directly from Proposition \ref{sat-lower bound}. We prove now the upper bound. Given a connected essential subsurface $Y\subset S$ let $\CO_{\Fill}(K,Y)$ be the number of $\Map(Y)$-orbits of curves $\gamma\subset Y$ which fill $Y$ and which satisfy $\iota(\gamma,\gamma) \leq K$, and note that, if $Y_1,\dots,Y_r$ are representatives for the finitely many mapping class group orbits of connected essential subsurfaces in $S$ we have that
$$\CO(K,S)=\sum_{i=1}^r\CO_{\Fill}(K,Y_i).$$
In particular, to bound the left side it suffices to give individual bounds for each summand. This is what we will do. In fact, since all cases are identical, and with the aim of simplifying the involved notation, we will limit ourselves to the upper bound for $\CO_{\Fill}(K,S)$.

The starting point is to recall that by Theorem \ref{good metric} and Corollary \ref{kor-inj-bound} we have for every filling curve $\gamma\subset S$ with $\iota(\gamma,\gamma)\le K$ a hyperbolic metric $\rho$ on $S$ with
$$\ell_\rho(\gamma)\le 4\sqrt{2\vert\chi(S)\vert\cdot K}\text{ and }\inj(S,\rho)\ge e^{-4\sqrt{2\vert\chi(S)\vert\cdot K}}.$$
We get thus from Proposition \ref{prop-upper-bound-thick} that for all $\delta>0$ there is $C$ with
$$\CO_{\Fill}(K,S)\le C\cdot n_S\left(e^{-4\sqrt{2|\chi(S)|\cdot K}},\delta\right)\cdot e^{e^{2\delta}4 \sqrt{2|\chi(S)| \cdot K}}$$
Plugging in the bound for $n_S(\cdot,\cdot)$ from Proposition \ref{prop-net}
 we get
\begin{align*}
\CO_{\Fill}(K,S)
& \le C\cdot n_S\left(e^{-4\sqrt{2|\chi(S)|\cdot K}},\delta\right)\cdot e^{e^{2\delta}4 \sqrt{2|\chi(S)| \cdot K}}\\
&\le C'\cdot\left\vert\log\left(e^{-4\sqrt{2|\chi(S)|\cdot K}}\right)\right\vert^N\cdot f(\delta)\cdot e^{e^{2\delta}4 \sqrt{2|\chi(S)| \cdot K}}\\
&= C'' \cdot K^{\frac N2}\cdot e^{e^{2\delta}4 \sqrt{2|\chi(S)| \cdot K}}
\end{align*}
where $C'$ is a constant depending only on the topology of the surface and $C''$ depends on the topology of $S$ and on $\delta$. The claim follows.
\end{proof}

\section{Further Comments}\label{sec-final}

\subsection{Surfaces with punctures} The upper bound immediately applies to any orientable surface $S$ of finite type with $\chi(S)<0$. However, the argument for the lower bound relies on the compactness of the space of projective currents, and the fact that length functions of curves extend continuously to finite-valued functions on the space of currents. When $S$ has cusps, these properties need not hold, although this difficulty can be circumvented by replacing each cusp with a boundary component. For simplicity, we have elected to present proofs for the lower bound only in the setting of closed surfaces, and to simply remark that with care, a similar bound can be obtained for non-closed surfaces as well. In this setting, let $X_\epsilon$ be a convex-cocompact hyperbolic surface whose convex core is homeomorphic to $S$ and has boundary of length $\epsilon$. In the arguments in section \ref{sec-lalley} and section \ref{sec-lower bound}, replace the Liouville current by the Patterson-Sullivan current, that is the current corresponding to the measure of maximal entropy for the recurrent part of the geodesic flow on $X$. Lalley's results still hold and the arguments still apply. 

\subsection{Sharpness} As there is still a gap between the lower and upper bounds, we next address the natural question of sharpness. In particular, we claim that the upper bound is in fact not sharp. Indeed, recall that the upper bound on $\ell_{\rho}(\gamma)$ from Theorem \ref{good metric} is of the form 
$$ \ell_{\rho}(\gamma) \leq 4 \sqrt{\iota(\gamma, \gamma)} \sqrt{\sum_{v} r_{v}^{2}},$$
where $r_{v}$ are the radii of a circle packing by hyperbolic disks on the surface equipped with the metric $\rho$. We obtain the conclusion of Theorem \ref{good metric} by observing that since the hyperbolic area of a disk of radius $r$ is larger than $\pi r^{2}$ and the disks in our circle packing have disjoint interiors, the sum on the right hand side of the above inequality is bounded above by the square root of the area of the entire surface. However, this bound is inefficient in that it includes the area of the surface which is not contained in any of the disks of the circle packing. Thus if one can estimate 

\begin{equation} \label{inefficient}
\text{Area}(S) - \sum_{D \in \CP} \text{Area}(D), 
\end{equation}
where $\CP$ is the set of disks in the packing, one can obtain an improved upper bound on $\CO(K,S)$. We claim that in fact the difference (\ref{inefficient}) can be bounded away from $0$. This follows from the fact that we can extend the filling curve $\gamma$ to a triangulation of $S$ which has degree bounded above by $12$. Thus, the dual circle packing will also have bounded degree. Then a Lemma of Rodin-Sullivan \cite{RodSul} (the ``ring lemma'') can be used to control the shape of each uniformized triangle in the triangulation on the surface equipped with the metric $\rho$, which can, in principle, be used to bound from below the area missed by the packing. 

On the other hand, we conjecture that the lower bound of Theorem \ref{main} is sharp:

\begin{conj}
$$ \lim_{K \rightarrow \infty} \frac{\log(\CO(K,S))}{\sqrt K}=\pi\sqrt{\vert\chi(S)\vert}$$
%
%$$ \lim_{K \rightarrow \infty} \frac{\CO(K,S) }{ e^{\pi \cdot \sqrt{|\chi(S)|\cdot K}}} = 1.$$
%

\end{conj}

\subsection{The size of the net}
We remark that the conclusion of Proposition \ref{prop-upper-bound-thick} is not sharp. Indeed,  one can show 
$$ n_{S}(\epsilon, \delta) \leq C(S, \delta) |\log(\epsilon)|^{\dim(\CT(S))}, $$
for $C$ a constant depending only on $S$ and $\delta$. However, improving the bound on $n_{S}(\epsilon, \delta)$ as above does not lead to an improved exponent in the upper bound for $\CO(K,S)$, so for simplicity we only sketch the proof here: 

One first argues that given $R>0$, there exists a constant $T=T(R,S, \delta)$ so that any ball of radius $R$ in $\CM(S)$ admits a $\delta$-net of size at most $T$. Then for $W$ a complete marking on $S$,  let $\CM^{W} \subset \CM(S)$ denote the surfaces for which $W$ is the shortest marking, and consider the map 
$$\lambda_{W}: \CM^{W} \rightarrow \mathbb{R}^{\dim(\CT(S))},$$
 sending $X$ to the tuple of logs of lengths of curves in $W$.

Then Theorem E of \cite{LRT} implies the existence of $C'=C'(S)$ so that for $X,Y \in \CM^{W}$,  
$||\lambda_{W}(X)- \lambda_{W}(Y)||_{\infty}$ small implies $d_{\text{simLip}}(X,Y) \leq C'$. It follows that for all $\delta$ sufficiently small, a $\delta$-net of $\text{Im}(\lambda_{W}) \subset \mathbb{R}^{\dim(\CT(S))}$ pulls back to a $C'$-net of $\CM^{W}$. Since $\text{Im}(\lambda_{W})$ restricted to $\CM^{\geq \epsilon}(S) \cap \CM^{W}$ lies in a Euclidean cube of volume roughly $| \log(\epsilon)|^{|W|}$, it then follows that 

$$ n_{S}(\epsilon, \delta) \leq T(C', S, \delta) \cdot M(S) \cdot |\log(\epsilon)|^{\dim(\CT(S))}, $$
where $M(S)$ is chosen to be much larger than the number of topological types of complete markings on $S$. 

\subsection{The Teichm{\"u}ller metric }Another advantage to the proof and conclusion of Proposition \ref{prop-upper-bound-thick} presented in Section $5$ is that both apply immediately to the $\epsilon$-thick part of $\CM(S)$ equipped with the Teichm{\"u}ller metric, as well as the symmetric Lipschitz metric. We compare this estimate to the work of Fletcher-Kahn-Markovic \cite{FKM}, which estimates the number of $\delta$-balls required to cover $\CM^{\geq \epsilon}(S)$ in the Teichm{\"u}ller metric, as a function of the genus of $S$. That is, they are primarily interested in fixing $\epsilon, \delta$ and letting $g \rightarrow \infty$, whereas Proposition \ref{prop-upper-bound-thick} is explicit in $\epsilon$, but not in the topology of the surface. 

We conclude by remarking that, using McMullen's K{\"a}hler-hyperbolic metric on Teichm{\"u}ller space and the fact that it is bi-lipschitz equivalent to the Teichm{\"u}ller metric, one can produce a bound on the order of $C(\epsilon,S)(1/\epsilon)^{\dim(\CT(S))}$ for the size of an $\epsilon$-net of $\CM^{\geq \epsilon}(S)$ in the Teichm{\"u}ller metric. By Wolpert's inequality, the Lipschitz metric is bounded above by the Teichm{\"u}ller metric, and hence this in turn produces a bound on the same order for $n_{S}(\epsilon, \epsilon)$. However, since we are interested in the $e^{-\sqrt{K}}$-thick part, such a bound will produce an exponent growing faster than $\sqrt{|\chi(S)|}$ as a function of $S$. To circumvent this, one might try to use Theorem $1.4$ of \cite{AGPS}, which produces a metric for which a given curve $\gamma$ with $\iota(\gamma, \gamma) \leq K$ has length $\leq C \cdot \sqrt{K}$ for some $C=C(S)$, and which is $1/\sqrt{K}$-thick. However, the only known bounds on the constant $C$ grow exponentially in $|\chi(S)|$, and this would significantly increase the coefficient of $\sqrt{K}$ in the exponent for the upper bound of $\CO(K,S)$. 

\subsection{Acknowledgements} This project began during the workshop \textit{Effective and algorithmic methods in hyperbolic geometry and free groups} at the Institute for Computational and Experimental Research in Mathematics (ICERM) in Providence, RI, and we thank ICERM for its hospitality. We also thank Jonah Gaster for pointing out that Theorem \ref{good metric} implies Corollary \ref{Gas}, and Sebastien Gouezel and Priyam Patel for helpful conversations. The first author was fully supported by NSF postdoctoral grant DMS-1502623.

\end{document}